\documentclass[11pt,a4paper]{article}
\usepackage[utf8]{inputenc}
\usepackage[T1]{fontenc}
\usepackage[english]{babel}
\usepackage[width=16.00cm, height=23.00cm]{geometry}

\usepackage{amsthm}
\usepackage{lipsum}
\usepackage{amsfonts}
\usepackage{graphicx}
\usepackage{epstopdf}
\usepackage{algorithmic}
\usepackage{mathtools}
\usepackage{amsmath}
\usepackage{amssymb}
\usepackage{amsfonts}
\usepackage{mathrsfs}
\usepackage{bbm}
\usepackage{xcolor}
\usepackage{hyperref}

\usepackage{cleveref}

\newtheorem{theorem}{Theorem}
\newtheorem{remark}{Remark}
\newtheorem{lemma}{Lemma}

\newcommand{\shape}{\Omega}

\let\svthefootnote\thefootnote
\newcommand\freefootnote[1]{%
	\let\thefootnote\relax%
	\footnotetext{#1}%
	\let\thefootnote\svthefootnote%
}

\title{Gâteaux semiderivative approach applied to shape optimization for contact problems}

\author{
	Nico Goldammer\footnote{Helmut-Schmidt-University / University of the Federal Armed Forces Hamburg, Holstenhofweg~85, 22043 Hamburg, 
		(\texttt{goldammer@hsu-hh.de}) }, 
	Volker H.~Schulz\footnote{Trier University, Universitätsring 15, 54296 Trier, Germany\newline
		(\texttt{volker.schulz@uni-trier.de}) }, 
	Kathrin Welker\footnote{TU Bergakademie Freiberg, Akademiestraße 6, 09599 Freiberg, Germany\newline
		(\texttt{Kathrin.Welker@math.tu-freiberg.de}) } 
}

\date{}

\begin{document}
	\maketitle
	

	\begin{abstract}
		Shape optimization problems constrained by variational inequalities (VI) are non-smooth and non-convex optimization problems. 
		The non-smoothness arises due to the variational inequality constraint, which makes it challenging to derive optimality conditions. 
		Besides the non-smoothness there are complementary aspects due to the VIs as well as distributed, non-linear, non-convex and infinite-dimensional aspects due to the shapes which complicate to set up an optimality system and, thus, to develop efficient solution algorithms. 
		In this paper, we consider Gâteaux semiderivatives in order to formulate
		optimality conditions. 
		In the application, we concentrate on a shape optimization problem constrained by the contact problem.
	\end{abstract}
	
	
	\paragraph{MSCcodes}
		49Q10, 49J40, 35Q93, 65K15 
	
	\freefootnote{\textbf{Founding:} This work has been partly supported by the German Research Foundation (DFG) within the priority program SPP1962/2 under contract numbers WE~6629/1-1 and SCHU~804/19-1.}

	\section{Introduction}
	
	Optimal control problems with constraints in the form of variational inequalities (VI) are challenging, since classical constraint qualifications for deriving Lagrange multipliers generally fail. 
	Therefore, not only the development of stable numerical solution schemes but also the development of suitable first order optimality conditions is an issue.
	By usage of tools of modern analysis, such as monotone operators in Banach spaces, significant results on properties of the solution operator of variational inequalities have been achieved since the 1960s (cf.~\cite{Br-1971,BrSt-1968,LiSt-1967}). 
	Comprehensive studies of variational inequalities and more references can be found in \cite{Glowinski-1984,KO-1988,KS-1980,Panagiotopoulos-1985}.
	The generic non-smoothness and non-convexity in the feasible set described by variational inequalities causes difficulties already in finite dimensional versions of the problem. 
	In fact, finite dimensional bilevel optimization (i.e., optimization with optimization problems in the constraints) is its own field of research since the 1970s (cf., e.g., \cite{Br-Gill-1973}) and has been generalized to mathematical programming with equilibrium constraints (MPECs) for the optimization of stationary systems of constrained problems in \cite{Ha-Pa-1988}. 
	For a survey on bilevel programming and MPECs see, e.g., \cite{LPR-1996}. 
	In \cite{SS-2000}, the authors concentrate on the typical complementarity structure of variational inequalities and derive a hierarchy of stationarity concepts (depending on constraint qualification conditions) for the more general problem class of mathematical programs with complementarity constraints (MPCCs). 
	During the last decade, these concepts have partly been transferred to respective concepts in function space in \cite{HMW-2012,HMW-2013,HK-2009}.
	The optimal control of variational inequalities that are posed in function space has been studied since the 1970s and necessary stationarity conditions have been derived by use of penalty and smoothing techniques and strengthened by the usage of instruments from convex analysis and differentiability, see, e.g., \cite{Barbu-1984,MP-1984,NST-2006}. 
	The conditions that a solution can be shown to verify have a complex structure and the problem to find candidates for solutions leads to a system of non-linear and non-smooth equations. 
	This demands for the development of numerical algorithms and a proper mathematical analysis on their convergence behavior, see, e.g., the discussion in \cite{KK-2002,OKZ-1998}. 
	
	In this paper, we consider shape optimization problems constrained by variational inequalities. 
	These problems are non-smooth and non-convex optimization problems. The non-smoothness arises due to the variational inequality constraint, which makes it challenging to derive optimality conditions. 
	Moreover, besides the non-smoothness there are complementarity aspects due to the VIs as well as distributed, non-linear, non-convex and infinite dimensional aspects due to the shapes which complicate to set up an optimality system.	
	In particular, one cannot expect for an arbitrary shape functional depending on solutions to VIs the existence of the shape derivative or to obtain the shape derivative as a linear mapping. 
	In addition, the adjoint state can generally not be introduced and, thus, an optimality system cannot be set up.
	A common way to handle the non-smoothness is to regularize the given problem, e.g., by replacing the $\max$-function by a smoothed approximation or by replacing certain functions with regularized version, and than consider the obtained regularized problem (see for example \cite{ChristofMeyerWaltherClason2018, SchielaWachsmuth2013, Mordukhovich2006, HintermullerKopacka2011}).
	{A mollification is also a tool that can be used to tweak non-smoothness as it is done in \cite{KovtunenkoKunisch2023}.
	They also discuss directional derivatives for shape optimization problems with VI. }
	{This paper aims at establishing a way to treat VI constrained shape optimization problems without the use of regularizations. 
We also aim at avoiding regularization techniques.  Therefore, we  use Gâteaux semiderivatives in order to obtain a useful kind of derivative such that we can deal with the unregularized contact problem.
}

	
	So far, there are only very few approaches in the literature to the problem class of VI constrained shape optimization problems. 
	In \cite{KO-1994}, shape optimization of 2D elasto-plastic bodies is studied, where the shape is simplified to a graph such that one dimension can be written as a function of the other. 
	In \cite[Chap.~4]{SokoZol}, shape derivatives of elliptic variational inequality problems are presented in the form of solutions to again variational inequalities. 
	In \cite{Myslinski-2001}, shape optimization for 2D graph-like domains are investigated. 
	Also \cite{LR-1991a,LR-1991b} present existence results for shape optimization problems which can be reformulated as optimal control problems, whereas \cite{DM-1998,G-2001} show existence of solutions in a more general set-up. 
	In \cite{Myslinski-2004,Myslinski-2007}, level-set methods are proposed and applied to graph-like two-dimensional problems. 
	Moreover, \cite{HL-2011} presents a regularization approach to the computation of shape and topological derivatives in the context of elliptic variational inequalities and, thus, circumventing the numerical problems in \cite[Chap.~4]{SokoZol}. 
	In \cite{Sturm-VI-2016}, the analysis of state material derivatives is significantly refined over \cite[Chap.~4]{SokoZol}.  
	All these mentioned problems have in common that one cannot expect for an arbitrary shape functional depending on solutions to VIs to obtain the shape derivative as a linear mapping (cf.~\cite[Example in Chap.~1]{SokoZol}). 
	In general, the shape derivative for VI-constrained problems fails to be linear with respect to the normal component of the vector field defined on the boundary of the open domain under consideration. 
	In order to circumvent the problems related to the non-linearity of the shape derivative and in particular the non-existence of the shape derivative of a VI constrained shape optimization problem, this paper concentrates on presenting a Gâteaux approach using Gâteaux semiderivatives. 

	This paper is structured as follows. 
	In \cref{sec:Model}, a VI constrained shape optimization problem is introduced and reformulated.
	Then, we derive the optimality system in the approach usually employed for Fréchet differentiable problems for VI constrained shape optimization in \cref{sec:OptSystem}, in order to prepare the reader to  \cref{sec:Application} in which we apply our Gâteaux approach to tackle a contact problem. 
	Within this context, we consider the Lagrangian associated with the contact problem and we use the concept of Gâteaux semiderivatives in order to calculate its derivative.
	This process enables us to generalize various objects and introduce a Gâteaux adjoint to the system under consideration.
	A conclusion is given in \cref{sec:conclusion}.

	
	\section{Problem formulation}
	\label{sec:Model}
	
	A main focus in shape optimization is in the investigation of shape functionals.
	A shape functional on an arbitrary shape space\footnote{Various shapes spaces have been extensively studied in recent decades. 
	In \cite{Kendall}, a shape space is just modelled as a linear (vector) space, which in the simplest case is made up of vectors of landmark positions.
	However, there is a large number of different shape concepts, e.g., plane smooth curves \cite{MichorMumford}, piecewise-smooth curves \cite{PryymakSuchanWelker2023}, surfaces in higher dimensions \cite{BauerHarmsMichor,MichorMumford2}, boundary contours of objects \cite{LingJacobs,RumpfWirth2}, multiphase objects \cite{WirthRumpf}, characteristic functions of measurable sets \cite{Zolesio},  morphologies of images \cite{DroskeRumpf},  and planar triangular meshes \cite{HerzogLoayzaRomero:2020:1}. The choice of the shape space depends on the demands in a given situation. There exists no common shape space  suitable for all applications.} 
	$\mathcal{U}$ is given by a function 
	$J\colon \mathcal{U} \to \mathbb{R}\text{, } \shape\mapsto J(\shape).$
	In general, a  shape optimization problem can be formulated by
	\begin{equation}
	\label{minproblem}
	\min_{\shape\in \mathcal{U}} J(\shape).
	\end{equation}
	Often, shape optimization problems are constrained by equations, e.g., equations involving an unknown function of two or more variables and at least one partial derivative of this function. The objective may depend on not only the shape $\Omega$ but also the  state variable $y$, where the state variable is the solution of the underlying constraint. 
	
	We consider a tracking-type shape optimization problem constrained by a variational inequality of the first kind, a so-called obstacle-type problem. Applications are manifold and arise, whenever a shape is to be constructed in a way not to violate constraints for the state solutions of partial differential equation depending on a geometry to be optimized. Just think of a heat equation depending on a shape, where the temperature is not allowed to surpass a certain threshold. This example is basically the model problem already considered in \cite{Luft2020} and that we are formulating in the following. In contrast to \cite{Luft2020}, which formulates an optimization approach based on the convergence of state, adjoint and shape derivative of the regularized problem to  limit objects, we do not consider regularized versions of the VI. 
	We consider a Gâteaux semiderivative approach in order to formulate an optimality system. We will see that this system is in line with the limit objects of \cite{Luft2020}.
	

	\paragraph{Model formulation.}
	
	{
		Let $\mathcal{X}\subset \mathbb{R}^n$ be an open bounded domain equipped with a sufficiently smooth boundary $\partial\mathcal{X}$. 
		This domain is assumed to be partitioned in an open subdomain $\mathcal{X}_\text{out}\subset\mathcal{X}$ and an open interior domain $\shape \subset\mathcal{X}$ with boundary $\Gamma:=\partial \shape$ such that $\mathcal{X}_\text{out}\sqcup \Gamma \sqcup \shape = \mathcal{X}$, where $\sqcup$ denotes the disjoint union. 
		The closure of $\mathcal{X}$ is denoted by $\bar{\mathcal{X}}$. 
		In the following, the boundary $\Gamma$ of the interior domain $\shape$ is called the interface.
		In the setting above, $\shape$ denotes the shape.
		In contrast to the outer boundary $\partial\mathcal{X}$, which is assumed to be fixed, the inner boundary is variable.
		If $\Gamma$ changes, then the subdomains $\shape,\mathcal{X}_\text{out}\subset \mathcal{X}$ change in a natural manner. Thus, one can consider $\mathcal{X}$ depending on $\Gamma$, i.e., $\mathcal{X} = \mathcal{X}(\Gamma)$. 
	}
	
	Let $\nu>0$ be an arbitrary constant. For the objective function
	$J(y,\shape):=\mathcal{J}(y,\shape)+ \mathcal{J}_\text{reg}(\Gamma)$
	with 
	\begin{align}\label{tracking}
	\mathcal{J}(y,\shape)& := \frac{1}{2}\int_{\mathcal{X}} \left(y - \bar{y}\right)^2 \mathrm{d} x,\\
	\mathcal{J}_\text{reg}(\Gamma)&:=\nu \int_{\Gamma} 1 \, \mathrm{d} s\label{regularization}
	\end{align}
	we consider 
	\begin{align}
	\min\limits_{\shape\in\mathcal{U}}\; J(y,\shape)\label{eq_minimization}
	\end{align}
	constrained by the  obstacle type variational inequality
	\begin{align}
	\label{VI_general}
	a(y,v-y)\geq \left<f,v-y\right> \quad\forall v\in K:=\{\theta \in H^1_0(\mathcal{X})\colon \theta(x)\leq \varphi(x) \text{ in }\mathcal{X}\},
	\end{align}
	where $y \in K$ is the solution of the VI, $f\in L^2(\mathcal{X})$ is explicitly dependent on the shape, $\left<\cdot,\cdot\right>$ denotes the duality pairing and $a(\cdot,\cdot)$ is a general strongly elliptic, i.e. coercive, symmetric bilinear form 
	\begin{align}\label{bilinearform}
	\begin{split}
	a\colon H_0^1(\mathcal{X})\times H_0^1(\mathcal{X}) &\rightarrow \mathbb{R} \\
	(y, v) &\mapsto \int_{\mathcal{X}} \underset{i,j}{\sum}   a_{i,j}\partial_iy  \partial_jv + \underset{i}{\sum}d_i( \partial_iy v + y \partial_iv) +  byv \; \mathrm{d} x
	\end{split}
	\end{align} 
	defined by coefficient functions $a_{i,j}, d_j, b\in L^\infty(\mathcal{X})$ fulfilling the weak maximum principle, {where $\partial_i$ denotes the partial derivative to the $i$-th component.}
	
	With the tracking-type objective $\mathcal{J}$ the model is fitted to data measurements $\bar{y}\in H^1(\mathcal{X})$.
	The second term $\mathcal{J}_\text{reg}$ in the objective function $J$ is a perimeter regularization. 
	In (\ref{VI_general}), $\varphi$ denotes an obstacle which needs to be an element of $L^1_{\text{loc}}(\mathcal{X})$ such that 
	the set of admissible functions $K$ is non-empty (cf.~\cite{SokoZol}).
	If additionally $\partial\mathcal{X}$ is Lipschitz and $\varphi \in H^1(\mathcal{X})$ with $\varphi_{\vert\partial\mathcal{X}} \geq 0$, then there is a unique solution to (\ref{VI_general}) satisfying $y\in H_0^1(\mathcal{X})$, given that the assumptions from above hold  (cf.~\cite{ito2000optimal,kinderlehrer1980introduction,troianiello2013elliptic}).
	Further, (\ref{VI_general}) can be equivalently expressed as
	\begin{align}\label{PDE}
	a(y,v)+(\lambda,v)_{L^2(\mathcal{X})} &= (f,v)_{L^2(\mathcal{X})} \quad \forall v\in H_0^1( \mathcal{X})
	\end{align}
	\begin{align}\label{VI_conditions}
	\begin{split}
	\lambda &\geq 0 \quad \text{in } \mathcal{X} \\
	y  &\leq \varphi \quad \text{in } \mathcal{X} \\
	\lambda(y-\varphi) &= 0 \quad \text{in } \mathcal{X}
	\end{split}
	\end{align}
	with $(\cdot,\cdot)_{L^2(\mathcal{X})}$ denoting the $L^2$-scalar product and $\lambda\in L^2(\mathcal{X})$.
	It is well-known, e.g., from \cite{kinderlehrer1980introduction}, that under these assumptions there exists a unique solution $y$ to the obstacle type variational inequality (\ref{VI_general}) and an associated Lagrange multiplier $\lambda$. 
	We assume this situation, which is also found in \cite{ItoKunisch_VI}, giving us $\lambda\in L^2(\mathcal{X})$. It can be easily verified that this in turn gives the possibility to summarize the conditions (\ref{VI_conditions}) equivalently into a single condition of the form 
	\begin{equation}
	\label{lambda}
	\lambda=\max\big(0,\lambda+\mathcal{C} (y-\varphi) \big)\quad \text{for any }	\mathcal{C}>0.
	\end{equation}
	In the following, we denote the active set corresponding to (\ref{PDE}) and (\ref{VI_conditions}) by $$A:=\{x\in\mathcal{X}\colon y-\varphi \geq 0 \}.$$

	After formulating the optimization problem under consideration and introducing the active set, we continue with the formulation of an optimality system for VI constrained shape optimization in the next section.
	{
		For this, we will focus on Gâteaux semiderivatives to derive the state and adjoint system and on the Eulerian derivative of the shape functional to set up the design equation.
		For further details on Gâteaux semiderivatives as well as on the Eulerian derivative we refer to \cite{SokolowskiZolesio1992,Delfour2020}.
		Given a functional $\mathcal{F}(\shape, y, z)$ depending on a shape $\shape$ and elements $y, z$ of topological spaces, we will denote the (total) Gâteaux semiderivative of $\mathcal{F}$ by $d^{G} \mathcal{F}$. 
		The notation $\partial_{y}^{G} \mathcal{F}(\shape, y, z)[\hat{y}]$ (and $\partial_{z}^{G} \mathcal{F}(\shape, y, z)[\hat{z}]$) means the Gâteaux semiderivative of $\mathcal{F}$ with respect to $y$ (and $z$) in direction $\hat{y}$ (and $\hat{z}$). 
		The Eulerian derivative of $\mathcal{F}$ at $\shape$ in direction $V$ is denoted by $\partial_{\shape}^G \mathcal{F}(\shape, y, z)[V]$.
		In order to be consistent with this notation we denote the material derivative by $d_m$.
	Given a functional $\mathcal{G}(y_1, \ldots, y_n)$ depending on elements $y_1,\dots,y_n$ of topological spaces we denote its Gâteaux semiderivative with respect to $i$-th component $y_i$ by $\partial_i^G \mathcal{G}$.
	}

	
	\section{Optimality system for VI constrained shape optimization}
	\label{sec:OptSystem}
	
	\def\Dom{\shape}
	\def\Hil{{\mathcal H}}
	\def\Lag{\mathscr{L}}
%
	In this section, we briefly discuss necessary optimality conditions for non-smooth shape optimization problems in our setting and terminology. Although shapes do not define a linear space, the shape derivative can be viewed as a directional derivative in the space of deformations  of the shape under investigation. This aspect is investigated further in \cite{Schmidt-Schulz-2023}, where a linear deformation space framework is established. We consider a space $Y$ as an appropriate vector space of deformations, such that the set $\mathcal{U}$ of admissible shapes $\shape$ is constructed as ${\cal S^\text{adm}}=\{T(\Dom^0) \colon T\in Y\}$, where $\Dom^0$ is a reference starting domain, which is assumed to be a subset of the open hold-all domain $D$. Thus, we can write the shape derivative of a functional $\mathcal{J}:\Omega\mapsto \mathbb{R}$ as
\[
d\mathcal{J}(\shape)[V] = \partial^G_W\mathcal{J}((I+W)(\shape))[V]
\]	
i.e., as a Gâteaux semiderivative with respect to $W$, where $I$ denotes the identity deformation and $W$ an arbitrary small deformation. We denote this derivative for convenience as $d^G\mathcal{J}(\shape)[V]$. Thus, for functionals with more arguments like above, we define
\begin{align*}
d^G\mathcal{F}(\shape, y(\shape),p(\shape))[V]&=d^G_W((I+W)(\shape), y((I+W)(\shape)),p((I+W)(\shape)))[V]\\
\partial^G_\shape\mathcal{F}(\shape, y(\shape),p(\shape))[V]&=d^G_W((I+W)(\shape), y(\shape),p(\shape))[V]
\end{align*}
Partial Gâteaux derivatives with respect to the other arguments are denoted as $\partial^G_y\mathcal{F}$ and 	$\partial^G_p\mathcal{F}$.

	We consider constrained shape optimization problems of the following form: 
	
	\begin{align}\label{sec3-gen-obj}
	&\min_{\shape\in \mathcal{U}}\quad J(\Dom,y)\\\label{sec3-gen-con}
	&\text{ s.t. } \, b(c(\shape,y),p)_\Dom=0\quad \forall\,p\in \Hil(\Dom)
	\end{align}
	Here, $\Hil(\Dom)$ is a Hilbert space defined on the shape $\Dom$ containing the state variable $y\in\Hil(\Dom)$ and $b(\cdot,\cdot)_\Dom$ is a bilinear and in $\Hil(\Dom)$ a coercive form. 
	Moreover, $\mathcal{U}$ is the set of admissible shapes, i.e., an appropriate shape space.
	We assume that the mapping $c$ is Gâteaux semidifferentiable, and that the constraint (\ref{sec3-gen-con}) defines a unique solution $y(\Dom,f)$ on any shape $\Dom$ under consideration. 
	
	Because $y(\Dom)$ is assumed to satisfy the constraint, we may write for arbitrary $p(\Dom)\in \Hil(\Dom)$
	\[
	J(\Dom,y(\Dom))=J(\Dom,y(\Dom))+b(c(\Dom,y(\Dom)),p(\Dom))_\Dom.
	\]
	In order to derive necessary conditions of optimality, we differentiate the right-hand side with respect to $\Dom$, i.e.~compute the shape derivative, and simplify the expressions by introducing the notation
	\[
	\Lag(\Dom, y, p):=J(\Dom,y)+b(c(\Dom,y),p)_\Dom,
	\]
	where we keep in mind the implicit dependence of $y,p$ on $\Dom$.  
	Thus, the chain rule yields
	\begin{align*}
	d^G \Lag(\Dom, y, p)[V]=\partial^E_{\Dom}\Lag(\Dom, y, p)[V]+\partial^G_y\Lag(\Dom, y, p)d_my
		+\partial^G_p\Lag(\Dom, y, p)d_{m}p
	\end{align*}
	for all $V\in Y$.

	Since $y$ satisfies the state equation (\ref{sec3-gen-con}) in variational form, which is linear in $p$, we observe 
	\begin{align}\label{eq:variational_state_eq}
	\partial^G_p\Lag(\Dom, y, p)d_m p=0.
	\end{align}
	Furthermore, we obtain
	\begin{align*}
	\partial^G_y\Lag(\Dom, y, p)D_{Hm}y &= \partial^G_y J(\Dom, y)d_{m}y+
		b(\partial^G_yc(\Dom,y)d_{m}y,p)_\Dom
	\end{align*}
	and, thus, we may obtain $p$ from the Gâteaux adjoint equation in variational form:
	\begin{align}\label{sec3-adjoint}
	\partial^G_y J(\Dom, y)\tilde{y}+
		b(\partial^G_y c(\Dom,y)\tilde{y},p)_\Dom = 0\quad \forall\,\tilde{y}\in \Hil(\Dom)
	\end{align}
	The solvability of the Gâteaux adjoint equation is in question in this rather general set-up. Therefore, we take it for granted now and show solvability, when confronted with the particular model problem as in the next section.
	Now, if $y$ satisfies the  state equation (\ref{sec3-gen-con})  and $p$ satisfies the Gâteaux adjoint equation (\ref{sec3-adjoint}), then the  Gâteaux shape semiderivative is given by
	\begin{align*}
	d^G \Lag(\Dom, y, p)[V]=\partial^{G}_{\Dom}\Lag(\Dom, y, p)[V].
	\end{align*}
	Nevertheless, it is a manually easier way to compute the  Gâteaux shape semiderivative of the full Lagrangian by employing shape and Gâteaux calculus and later on eliminate expressions relating to the state and  Gâteaux  adjoint equation, as exemplified in the next section.

	First we formulate necessary conditions of optimality for the minimization problem.
	\begin{theorem}
		\label{nec_uncon}
		We assume that the function $J\colon X\to \mathbb{R}$, where $X$ is a Banach space, is Gâteaux semidifferentiable and that $\hat{x}\in X$ is a local minimum of $J$. Then,  there holds
		\[
		{d^G }J(\hat{x})[v]\ge 0 \qquad \forall\, v\in X.
		\] 
	\end{theorem}
	
	\begin{proof}
		Due to the fact that $\hat{x}$ is the minimum, there holds $J(z)\ge J(\hat{x})$ for all $z\in X$. We choose in particular $z:=\hat{x}+tv$ for an arbitrary $v\in X$ and $t>0$. From this, we conclude
		\[
		\frac{1}{t}\left(J(\hat{x}+tv)- J(\hat{x})\right)\ge 0\, ,
		\]
		and thus we obtain the assertion by using the definition of the Gâteaux semideri\-vative.
	\end{proof}

	From Theorem \ref{nec_uncon}, we conclude now the necessary condition of optimality for an optimal shape $\shape$ as
	\begin{equation}\label{Nshape_nec}
	\partial^G_\shape \Lag(\Dom, y, p)[V]\ge 0\quad \forall\, V \in Y,
	\end{equation}
	where $y$ satisfies the state equation (\ref{sec3-gen-con})  and $p$ satisfies the Gâteaux adjoint equation (\ref{sec3-adjoint}). 
	This means that we can  observe from theorem \ref{nec_uncon} that the Gâteaux (shape) semiderivative can be used to characterize necessary optimality conditions.	For further exploration on first order optimality conditions in the Gâteaux framework, we refer to the literature, e.g.,  \cite{Br-Gill-1973} 
		
	In many cases, as is demonstrated in the next section, the Gâteaux shape semiderivative is continuous
, although the constraints of the shape optimization problem are only semismooth. Then, the necessary condition is just the usual homogeneity of the shape derivative. In this case, the (Gâteaux) shape (semi)derivative can be used in order to define a descent direction for algorithmical purposes. Nevertheless, finding a descent direction from the Gâteaux  semiderivative is a challenge in general.

	In the next section, we study weak formulations of elliptic problems. These are typically formulated in the Sobolev space $H^1(\Dom)$ of weakly differentiable $L^2$-functions. For the standard elliptic heat-equation-type problem, the solution is mostly in $H^2(\Dom)\subset H^1(\Dom)$ and, thus, their material derivative again in $H^1(\Dom)$.
However, in the context of variational inequalities, the solution is only piecewise $H^2(\Dom)$, which means that material derivatives cannot be used as test functions like in (\ref{eq:variational_state_eq}). A similar problem arises in discontinuous Galerkin approximations, from where we borrow the notion of a ``broken'' Sobolev space here, which is analyzed in detail in \cite{CARSTENSEN-2016}. This concept is based on a disjoint partitioning $\Dom_h$ of open subsets $\mathcal{K}\subset\Dom$ with Lipschitz boundaries such that
	$\overline{\cup_{\mathcal{K}\in\Dom_h}\mathcal{K}}=\Dom$. Then one defines
	\[
	H^1(\Dom_h):=\{\shape\in L^2(\Dom)\, :\, \shape|_\mathcal{K}\in H^1(\mathcal{K}), \mathcal{K}\in \Dom_h\}
	\]
	In \cite{CARSTENSEN-2016}, this space is used as test space and shown that a resulting weak formulation of the standard elliptic problem exists which inherits stability and, thus, existence of a unique solution. Thus, we mean this more general weak formulation in the following, whenever a test function is used, which is only piecewise $H^1$.
	
	
	\section{Application to shape optimization for contact problems}
	\label{sec:Application}

		In this section, we apply the Gâteaux semiderivative approach to  a contact problem. In particular, we introduce the  Lagrangian of the problem, compute its Gâteaux semiderivative and focus  on a Gâteaux adjoint associated with the system under consideration.

	\paragraph{Gâteaux adjoint equation.}
	
	Since the perimeter regularization (\ref{regularization}) is only used due to technical reasons to overcome ill-posedness of inverse problems (cf., e.g.,~\cite{Burger-2004}) and does not influence the adjoint system, we omit it for our investigations in the following. Thus, we consider the (reduced) Lagrangian function to the minimization of (\ref{tracking}) constrained by 
	\begin{align}\label{PDElambda}
	a(y,v)+(\max\big(0,\lambda+\mathcal{C} (y-\varphi) \big),v)_{L^2(\mathcal{X})} &= (f,v)_{L^2(\mathcal{X})} \quad \forall v\in H_0^1( \mathcal{X}),
	\end{align}
	which is given by 
	\begin{equation}
	\label{Lagrangian}
	\begin{split}
	\Lag(\shape,y,v) = 
	&\, \frac{1}{2} \int_{\mathcal{X}}(y-\bar{y})^{2}\, \mathrm{d} x-a(y,v) -\int_{\mathcal{X}} fv\, \mathrm{d} x\\ &+\int_{\mathcal{X}} \max \{0, \lambda+\mathcal{C} (y-\varphi)\}v \,  \mathrm{d} x,
	\end{split}
	\end{equation}
	to formulate the Gâteaux adjoint equation to the model problem (\ref{eq_minimization})--(\ref{VI_general}) by computing $\partial^G_y\Lag(\shape,y,v)$.

	In order to compute $\partial^G_y\Lag(\shape,y,v)$, we consider a variation of $y$. 
	Let $t>0$ and  $\tilde{y}\in H^1_0(\mathcal{X})$. 
	Then, we get
	\begin{align*}
	\partial^G_y \Lag(\shape,y,v)[\tilde{y}]
	&=\partial^G_t \Lag(\shape,y+t\tilde{y},v)\\
	=& \int_{\mathcal{X}}(y-\bar{y})  \tilde{y} \, \mathrm{d} x  - a(\tilde{y},v) +\int_{\mathcal{X}} \partial^G_t \left(\max \{0, \lambda+\mathcal{C} (y + t \tilde{y}-\varphi)\}v \right)\,  \mathrm{d} x.
	\end{align*}
	Using the chain rule we obtain 
	\begin{align*}
	\partial^G_t&\left(\max \{0, \lambda+\mathcal{C} (t \tilde{y}-\varphi)\}v \right)\\
	&=\partial^G_y(\max \{ 0, \cdot \})(\lambda + \mathcal{C} ( y - \varphi))  \,\,[\partial^G_t(\lambda+\mathcal{C} (t \tilde{y}-\varphi))\, v]\\
	&=\partial^G_y(\max \{ 0, \cdot \})(\lambda + \mathcal{C} ( y - \varphi))  \,\,[\mathcal{C} \tilde{y}v].
	\end{align*}

	
	{
	Then the Gâteaux semiderivative yields 
	\begin{align}
	\label{equality_max}
	\begin{split}
		\partial^G_t &\left(\max \{0, \lambda+{\mathcal{C}} (y + t \tilde{y}-\varphi)\}v \right)\\
		&=	\begin{cases}
		\mathcal{C}\tilde{y}v	,			& \lambda+{\mathcal{C}} (y + t \tilde{y}-\varphi) > 0, \\
		\max\{0,\mathcal{C}\tilde{y}v\},	& \lambda+{\mathcal{C}} (y + t \tilde{y}-\varphi) = 0,\\
		0	,			& \lambda+{\mathcal{C}} (y + t \tilde{y}-\varphi) < 0.	\end{cases}\\
		&=	\begin{cases}
		\mathcal{C}\tilde{y}v	,			& \text{  in } A, \\
		\max\{0,\mathcal{C}\tilde{y}v\},	& \text{  in } \{x \in \mathcal{X}: y=\varphi \}, \\
		0	,								& \text{  in } \{x \in \mathcal{X}: y<\varphi \}.	\end{cases}
	\end{split}
	\end{align}
	
	\noindent
	This results in
	\begin{align*}
		d_t^G\Lag(\shape,y+t\tilde{y},v)
		=& \int_{\mathcal{X}}(y-\overline{y})  \tilde{y} \, \mathrm{d} x  
		-a(\tilde{y},v) 
		+ \int_{\mathcal{X}} \mathbbm{1}_A\, \mathcal{C}\tilde{y}v \,  \mathrm{d} x \\
		&+ \int_{\{x \in \mathcal{X}:\  y=\varphi \}} \max\{0,\mathcal{C}\tilde{y}v\} \, \mathrm{d} x ,
	\end{align*}
	where $ \mathbbm{1}_A$ denotes the indicator function on the active set $A$.
	As a result, the Gâteaux adjoint equation is given in its weak form by
	\begin{align*}
		\int_{\mathcal{X}}(y-\overline{y})  \tilde{y} \, \mathrm{d} x  -a(\tilde{y},v) 
		=& - \int_{\mathcal{X}} \mathbbm{1}_A\,  \mathcal{C}v\tilde{y} \, \mathrm{d} x\\
		&- \int_{\{x \in \mathcal{X}:\ y=\varphi \}}  \max\{0,\mathcal{C}\tilde{y}v\} \, \mathrm{d} x 
		\quad\forall\, \tilde{y}\in H_0^1(\mathcal{X}). \nonumber
	\end{align*}

	\begin{remark}
		The additional integral over $\max\{0,\mathcal{C}\tilde{y}v\}$ on $\{x \in \mathcal{X}: y=\varphi \}$ seems to hold further challenges.
		However, numerical experiments have shown, that this expression never holds any numerical significance; see \cite{Suchan24}.
		As a consequence, we are going to neglect it and assume that $\max\{0,\mathcal{C}\tilde{y}v\}$ is of measure zero, and thus
		\begin{align}
		\label{eq:NewtonAdjoint}
		\int_{\mathcal{X}}(y-\overline{y})  \tilde{y} \, \mathrm{d} x  -a(\tilde{y},v) 
		=& - \int_{\mathcal{X}} \mathbbm{1}_A\,  \mathcal{C}v\tilde{y} \, \mathrm{d} x
		\end{align}
		holds.
		This results in the necessity of a safe guard technique, since we can not be sure that the Gâteaux semiderivative provides a descent direction.
		Such a technique is presented in \cite{Luft2020}.
	\end{remark}
}

	\paragraph{Gâteaux semiderivative of the (full) Lagrangian}
	In order to set up the optimality system to the model problem (\ref{eq_minimization})--(\ref{VI_general}), we need the Gâteaux semiderivative of the (full) Lagrangian
	$\Lag_{\text{full}}(y, \shape,v) = \Lag(y, \shape,v)   + \mathcal{J}_\text{reg}(\Gamma),$
	where $\Lag$ denotes the (reduced) Lagrangian (\ref{Lagrangian}).
	The Gâteaux semiderivative of $\Lag_{\text{full}}$ is given by the sum of the  Gâteaux semiderivative of the (reduced) Lagrangian (\ref{Lagrangian}) and the shape derivative of $\mathcal{J}_{\text{\emph{reg}}}$. Standard calculation techniques yield the shape derivative of $\mathcal{J}_{\text{\emph{reg}}}$, which is given by $d^E\mathcal{J}_{\text{\emph{reg}}}(\Gamma)[\cdot]= \nu\int_{\Gamma}\kappa\left<\cdot,n\right> \mathrm{d} s$ with $\kappa:=\text{div}_{\Gamma}(n)$ denoting the mean curvature of~$\Gamma$. 
	
	The next lemma gives the Gâteaux semiderivative of the Lagrangian.
	
	\begin{lemma}
		\label{le:NewtonShapeDerivativeOfL}
		Let $ \varphi \in H^{2}(\mathcal{X})$, $f \in L^{2}(\mathcal{X})$, $\bar{y} \in H^{1}(\mathcal{X})$, $v \in H^1_0(\mathcal{X})$ and $\lambda \in L^2(\mathcal{X})$.
		Then,
		\begin{equation}
		\label{Example:ShapeDeriv}
		\begin{split}
		d^G\mathscr{L}(\shape,y,v)[V]
		=& \int_{\mathcal{X}} \operatorname{div}(V)\left[\frac{1}{2}(y-\bar{y})^{2}+\nabla y^{\top} \nabla v - fv\right] \, \mathrm{d} x\\
		&- \int_{\mathcal{X}} (y - \bar{y}) \nabla\bar{y}^{\top}V \, \mathrm{d} x+ \int_{\mathcal{X}}  \nabla f^{\top}V\,v \,  \mathrm{d} x\\
		&- \int_{\mathcal{X}} \sum_{i,j}a_{i,j}\left( -\partial^G_j v\sum_l \partial^G_l y \,\partial^G_i V_l - \partial^G_i y \sum_{l}\partial^G_lv\,\partial^G_jV_l\right) \, \mathrm{d} x \\
		&- \int_{\mathcal{X}} \sum_i d_i\left( -v\sum_l\partial^G_ly\,\partial^G_iV_j - y\sum_l\partial^G_lv\,\partial^G_iV_j \right) \, \mathrm{d} x \\
		&	+ \int_{A} (\varphi - \bar{y})\nabla \varphi^\top V \, \mathrm{d} x.
		\end{split}
		\end{equation}
	\end{lemma}
	
	To prove this Lemma we need the following result 
	
	\begin{lemma}
		\label{le:D_N(int p_t)}
		Let the family $F_t$ of transformations be differentiable in the usual sense.
		We define the domain integral $J(\shape)=\int_\shape g\, \mathrm d x$ for a function $g\colon  \shape \rightarrow \mathbb{R}$. Then, we have
		\begin{align*}
		d^E J(\shape)[V]
		= \int_{\shape} d_m {g} + \operatorname{div}V \, g\, \, \mathrm d x.
		\end{align*}
	\end{lemma}
	
	\begin{proof}
		A proof is given in \cite[Theorem 4.11]{Welker_diss} for the the case that $F_t$ is the perturbation of identity.
		If we are instead using the velocity method we have the following
		\begin{align*}
		\left.\left(\partial^G_t\right)_{|_{t=0^+}}\left(\int_{\Omega_{t}} g_t \mathrm{d} x_{t}\right)\right.
		& =\left.\left(\partial^G_t\right)_{|_{t=0^+}}\left(\int_{\Omega}\left(g_t \circ F_{t}\right) \cdot \operatorname{det}\left(d^E F_{t}\right) \mathrm{d} x\right)\right.\\ 
		& =\left.\int_{\Omega} 
		\left(\partial^G_t\right)_{|_{t=0^+}}\left(\left(g_t \circ F_{t}\right) \cdot \operatorname{det}\left(d^E F_{t}\right)\right) \right.
		\mathrm{d} x\\
		& =\left.\int_{\Omega} 
		\left(\partial^G_t\right)_{|_{t=0^+}}( g_t \circ F_{t} ) \operatorname{det}(d^E F_{t})
		+ ( g_t \circ F_{t} ) \frac{\partial^E}{\partial^E t}_{|_{t=0^+}} \hspace*{-.3cm}\left(\operatorname{det}(d^E F_{t}) \right) \right.
		\mathrm{d} x\\
		& =\left.\left.\int_{\Omega} 
		\left(\partial^G_t\right)_{|_{t=0^+}}( g_t \circ F_{t} )\right.\operatorname{det}(G(0))
		+ g \frac{\partial^E}{\partial^E t}_{|_{t=0^+}} \hspace*{-.3cm} \left(\operatorname{det}(G_{t}) \right)\right.
		\mathrm{d} x\\
		& =\left.\int_{\Omega} 
		d_m  g
		+ g \left(\partial^G_t\right)_{|_{t=0^+}} \hspace*{-.3cm} \left(\operatorname{det}(G_{t}) \right)\right.
		\mathrm{d} x\\
		& =\left.\int_{\Omega} 
		d_m  g
		+ g \left(\det(G(t)) \operatorname{tr}\left( G^{-1}(t) \left(\partial^G_t\right)_{|_{t=0^+}} G(t)  \right) \right)\right.
		\mathrm{d} x\\
		& =\int_{\Omega} 
		d_m  g
		+ g \det(G(0)) \operatorname{tr}\left( G^{-1}(0) d^EV  \right)
		\mathrm{d} x\\
		& =\int_{\Omega} 
		d_m  g
		+ g  \operatorname{tr}\left( d^EV  \right)
		\mathrm{d} x\\
		& =\int_{\Omega} 
		d_m  g
		+ g  \operatorname{div}\left(V  \right)
		\mathrm{d} x
		\end{align*}
		
	\end{proof}
	
	We are now able to prove \cref{le:D_N(int p_t)}.
	\begin{proof}
		For an easier understanding and notation purpose we define
		\begin{align*}
		\chi(y,v) \coloneqq \underset{i,j}{\sum}   a_{i,j}\partial^G_iy  \partial^G_jv + \underset{i}{\sum}d_i( \partial^G_iy v + y \partial^G_iv) +  byv
		\end{align*}
		such that
		\begin{align*}
		a(y,v) = \int_{\mathcal{X}} \chi(y,v) \, \mathrm{d} x.
		\end{align*}
		Let 
		\begin{equation}
		\label{G}
		\begin{split}
		G(\shape&,y,v)[V]\\
		\coloneqq& \int_{\mathcal{X}}
		d_m\left(\frac{1}{2}(y - \bar{y})^2 - \chi(y,v) + fv + \max\{0, \lambda + \mathcal{C} ( y- \varphi)\}v\right) \\
		& \,\,\,  \,\,\,+ \operatorname{div}(V)\left[\frac{1}{2}(y-\bar{y})^{2}+\chi(y,v) -fv\right] \, \mathrm{d} x.
		\end{split}
		\end{equation}
		We consider a variation $\shape_t=F_t(\shape)$ of $\shape$ in the following. Since $\mathcal{X}$ depends on $\shape$, we also use the notation $\mathcal{X}_t:=F_t(\mathcal{X})$. 
		We get
		\begin{align*}
		G&(\shape,y,v)[V]\\
		=& \int_{\mathcal{X}} \operatorname{div}(V)\left[\frac{1}{2}(y-\bar{y})^{2}+\chi(y,v) - fv\right] \\
		&\phantom{\int_{\mathcal{X}}}\,\, + 
		d_m\left(\frac{1}{2}(y - \bar{y})^2\right)
		- d_m(\chi(y,v))
		+ d_m(fv)\\
		&\phantom{\int_{\mathcal{X}}}\,\,
		+ d_m(\max\{0, \lambda + {\mathcal{C}} ( y- \varphi)\}v) \, \mathrm{d} x	\\
		=& \int_{\mathcal{X}} \operatorname{div}(V)\left[\frac{1}{2}(y-\bar{y})^{2}+\chi(y,v) - fv\right] \\
		&\phantom{\int_{\mathcal{X}}}\,\,+ (y - \bar{y})d_my -(y - \bar{y})d_m\bar{y}+ vd_mf + f d_mv  \\
		&\phantom{\int_{\mathcal{X}}}\,\, - d_m\left( \sum_{i,j} a_{i,j}
		\partial^G_i y  \partial^G_j v 
		+ \sum_i d_i
		( \partial^G_i y  v + y \partial^G_i v) 
		+  b  y v \right) \\
		&\phantom{\int_{\mathcal{X}}}\,\, +d_m(\max\{0, \lambda + {\mathcal{C}} ( y- \varphi)\})v + \max\{0, \lambda + {\mathcal{C}} ( y- \varphi)\} d_mv \, \mathrm{d} x \\
		=& \int_{\mathcal{X}} \operatorname{div}(V)\left[\frac{1}{2}(y-\bar{y})^{2}+\chi(y,v) - fv\right] \\
		& \phantom{\int_{\mathcal{X}}}\,\, + (y - \bar{y})d_my -(y - \bar{y})d_m\bar{y}+ vd_mf + f d_mv  \\
		&\phantom{\int_{\mathcal{X}}}\,\, - \chi(d_my,v) + \chi(y,d_mv)  \\
		&\phantom{\int_{\mathcal{X}}}\,\, -\sum_{i,j}a_{i,j}\left( -\partial^G_j v \sum_{l}\partial_l y\, \partial^G_i V_l- \partial^G_i y \sum_{l}\partial^G_lv\,\partial^G_jV_l\right) \\
		&\phantom{\int_{\mathcal{X}}}\,\, -\sum_i^nd_i(\partial^G_iyd_mv + d_my \,\partial^G_i v)  \\
		&\phantom{\int_{\mathcal{X}}}\,\, + d_m(\max\{0, \lambda + {\mathcal{C}} ( y- \varphi)\})v + \max\{0, \lambda + {\mathcal{C}} ( y- \varphi)\} d_mv \, \mathrm{d} x 
		\end{align*}
		as well as
		\begin{align*}
		&\lim_{t\searrow0}\frac{\mathscr{L}(\shape,y,v)-\mathscr{L}\left(\shape_t,y,v\right)}{t} \\
		&=
		\lim_{t\searrow0}
		\dfrac{\frac{1}{2} \int_{\mathcal{X}}(y-\bar{y})^{2} \, \mathrm{d} x - \frac{1}{2} \int_{\mathcal{X}_t} (y_t-\bar{y_t})^{2} \, \mathrm{d} x} {t}
		- \dfrac{\int_{\mathcal{X}} \chi(y,v) \, \mathrm{d} x - \int_{\mathcal{X}_t} \chi(y_t,v_t) \, \mathrm{d} x}{t} \\
		&\phantom{= \lim\left|\right.}\,\,\, - \dfrac{\int_{\mathcal{X}} fv \, \mathrm{d} x - \int_{\mathcal{X}_t} f_tv_t \, \mathrm{d} x}{t}
		\\
		&\phantom{= \lim\left|\right.}\,\,\, + \dfrac{\int_{\mathcal{X}} \max (0, \lambda + \mathcal{C} ( y - \varphi))v \, \mathrm{d} x - \int_{\mathcal{X}_t} \max (0, \lambda_t + c_t( y_t - \varphi))v_t \, \mathrm{d} x}{t} .
		\end{align*}
		Combining the Gâteaux semiderivative of the maximum function with the equality (\ref{equality_max}) and the assumption  that $\partial A$ is a measure zero set, gives 
		\begin{align*}
		\int_{\mathcal{X}} d_m(\max\{0, \lambda + {\mathcal{C}} ( y- \varphi)\})v \, \mathrm{d} x= \int_{\mathcal{X}} \mathbbm{1}_A (d_m\lambda+{\mathcal{C}}  (d_my - d_m\varphi))v \, \mathrm{d} x.
		\end{align*}
		
		In addition, we know that
		\begin{align*}
		d_m\bar{y} =\nabla  \bar{y}^\top V\quad \text{  and }\quad d_mf =\nabla f^\top V
		\end{align*}
		if we assume $\bar{y}$ and $f$  independent of the shape.
		Thus, thanks to the state equation~(\ref{PDElambda}) and the Gâteaux adjoint (\ref{eq:NewtonAdjoint}) we get
		\begin{align*}
		G(\shape,y,v)[V]
		=& \int_{\mathcal{X}} \operatorname{div}(V)\left[\frac{1}{2}(y-\bar{y})^{2}+\nabla y^{\top} \nabla v - fv\right] \\
		& \phantom{\int_{\mathcal{X}}}\,\,  - (y - \bar{y})\nabla \bar{y}^\top V +v\nabla  f^\top V\\
		& \phantom{\int_{\mathcal{X}}}\,\, - \sum_{i,j}a_{i,j}\left( -\partial^G_j v\sum_l \partial^G_l y \,\partial^G_i V_l - \partial^G_i y \sum_{l}\partial^G_lv\,\partial^G_jV_l\right) \\
		& \phantom{\int_{\mathcal{X}}}\,\, - \sum_i^n d_i\left( -v\sum_{l}\partial^G_ly\partial^G_iV_j - y\sum_{l}\partial^G_lv\partial^G_iV_j \right) \mathrm{d} x \\
		&+\int_A (y - \bar{y})d_my\, \mathrm{d} x. 
		\end{align*}
		In the active set $A$, we have $y=\varphi$. 
		
		Moreover, $d_m\varphi =\nabla \varphi^\top V$ yields. 
		Thus, the integral over the active set is given by
		\begin{align*}
		\int_A (y - \bar{y})d_my\, \mathrm{d} x=\int_A (\varphi - \bar{y})\nabla \varphi^\top V \, \mathrm{d} x.
		\end{align*}
		Combining this with lemma~\ref{le:D_N(int p_t)} we see
		\begin{align*}
		\lim_{t\searrow0}&\frac{\mathscr{L}(\mathcal{X})-\mathscr{L}\left(\mathcal{X}_{t}\right)}{t} = G(\shape,y,v)[V].
		\end{align*}
		Therefore, $d^G\mathscr{L}$  is given by $G$.
	\end{proof}
	
	\begin{remark}
		It is worth mentioning that the Gâteaux adjoint equation~(\ref{eq:NewtonAdjoint}) and the Shape derivative given in lemma~\ref{le:NewtonShapeDerivativeOfL} are the limit object in \cite[Theorem 3.3]{Luft2020} and \cite[Theorem 3.5]{Luft2020}, respectively. Consequently, 
		if we consider the special case 
		$
		a(y,v): = \int_{\mathcal{X}} \nabla y^\top \nabla v \, \mathrm{d} x
		$
		as in \cite[section 4]{Luft2020},
		the Lagrangian is given by 
		$
		\Lag(y, \shape,v) =  \int_{\mathcal{X}}\frac{1}{2}(y-\bar{y})^{2}-\nabla y^\top \nabla v -fv+ \max \{0, \lambda+\mathcal{C} (y-\varphi)\}v \, \mathrm{d} x.
		$
		Then, lemma~\ref{le:NewtonShapeDerivativeOfL} yields the Gâteaux  semiderivative
		\begin{align*}
		d^G\mathscr{L}&(\shape,y,v)[V]\\
		= \int_{\mathcal{X}}&-(y-\bar{y})\nabla \bar{y}^\top V-\nabla y^\top (\nabla V^\top+\nabla V)\nabla p\\ &+\operatorname{div}(V)\left[\frac{1}{2}(y-\bar{y})^{2}+\nabla y^{\top} \nabla v - fv\right] \mathrm{d} x+\int_A(\varphi-\bar{y})\nabla\varphi^\top V\mathrm{d}  x ,
		\end{align*}
		which confirms the limit object given in \cite[equality (43)]{Luft2020}.
	\end{remark}

	\paragraph{Gâteaux optimality system}

	Here, we summarize the  optimality conditions. 
	For a solution shape $\shape$ to problem (\ref{eq_minimization})--(\ref{VI_general}), there holds the Gâteaux adjoint variational equation (\ref{eq:NewtonAdjoint}). 
	Since the Gâteaux  semiderivative $d^G \mathscr{L}(\shape,y,v)[V]$  given in lemma~\ref{le:NewtonShapeDerivativeOfL} is continuous in $V$, we obtain from (\ref{Nshape_nec}) the following necessary condition for the optimal shape $\shape$:
	\begin{align*}
	0=d^G \mathscr{L}(\shape,y,v)[V]\qquad \forall\, V\in H^1(\shape,\mathbb{R}^n)
	\end{align*}
	The Gâteaux adjoint equation, this necessary condition and the state equation (\ref{VI_general}) define together a set of equations, which is used for the computation of the solution in \cite{Luft2020}, where a perturbation approach is used for construction of $d^G \mathscr{L}(\shape,y,v)[V]$. 
	We observe also that $d^G \mathscr{L}(\shape,y,v)[V]$ is an integral on $\mathcal{X}$, where the integrand is  Gâteaux semidifferentiable with respect to $\shape$ and which lacks standard differentiability only at the the boundary of the active set $A$, which is a set of Lebesgue measure zero. 
	Thus, $d^G \mathscr{L}(\shape,y,v)[V]$ is a {Gâteaux} shape {semi}derivative and can therefore be used, in order to define a descent direction by employing an appropriate scalar product.
	In \cite{Luft2020}, the same expression has been derived in a perturbation approach, which necessitates a safeguard technique. 
	To be more precise,  there are used limit objects of the regularized problem in \cite{Luft2020} but  it is not guaranteed that the limit object of the regularized shape derivative yields a gradient in the classical sense in such a way that it is a descent direction.
	Therefore,  a safeguard technique that checks if the limit object provides a descent direction is necessary in the algorithm based on the approach in \cite{Luft2020}.

	
	\section{Conclusion}
	\label{sec:conclusion}

		In this paper, the concept of the Gâteaux semiderivative is used to generalize some objects and methods for shape calculus.
		One of the major advantages of the Gâteaux approach is that one no longer needs to regularize the variational inequality constraint in optimization problems. Moreover,  a limit process as in \cite{Luft2020} can be avoided.
		Considering Gâteaux semiderivatives results in an approach for the contact problem. 
		In addition, this paper explains the limiting expression for the shape derivative given in \cite{Luft2020} now as an expression derived from a Gâteaux adjoint. 

	\section*{Acknowledgments}
	This work has been partly supported by the German Research Foundation (DFG) within the priority program SPP1962/2 under contract numbers WE~6629/1-1 and SCHU~804/19-1. 

	\bibliographystyle{siamplain}
	\bibliography{references}
\end{document}